\documentclass[12pt,a4paper,reqno]{amsart} 
\usepackage[mathcal]{euscript} 
\usepackage{amssymb}
\usepackage[bookmarks=false]{hyperref}

\newtheorem{thm}{Theorem}[section] 
\newtheorem{prop}{Proposition}[section] 
\newtheorem{cor}{Corollary}[section]
\newtheorem{lemma}{Lemma}[section]
 
\theoremstyle{remark}
\newtheorem{rem}{Remark}[section]      
\newtheorem{ex}{Example}[section] 

\newcommand{\mr}{{\mathbb R}}
\newcommand{\mn}{{\mathbb N}}
\newcommand{\mc}{{\mathbb C}}
\newcommand{\mz}{{\mathbb Z}}

\newcommand{\eps}{\varepsilon}

\newcommand{\hil}{\mathcal{H}}
\newcommand{\dist}{\operatorname{dist}}
\newcommand{\num}{\operatorname{Num}}
\newcommand{\bdd}{\mathcal{B}}
\newcommand{\cld}{\mathcal{C}}
\newcommand{\dom}{\operatorname{Dom}}
\renewcommand{\Im}{\operatorname{Im}}
\renewcommand{\Re}{\operatorname{Re}}

\begin{document}

\title[An eigenvalue estimate and its applications]{An eigenvalue estimate 
and its application  to non-selfadjoint Jacobi and Schr\"odinger operators.}
  
\begin{abstract}
For bounded linear operators $A,B$ on a Hilbert space $\hil$ we show the validity of the estimate
$$ \sum_{\lambda \in \sigma_d (B)} \dist(\lambda, \overline{\num}(A))^p \leq \| B-A \|_{\mathcal{S}_p}^p, \quad p \geq 1,$$
and apply it to recover and improve some Lieb-Thirring type inequalities for non-selfadjoint Jacobi and Schr\"odinger operators.
\end{abstract}

\author{Marcel Hansmann}
\address{Institute of Mathematics, Clausthal University of Technology, 38678 Clausthal-Zellerfeld, Germany}
\email{hansmann@math.tu-clausthal.de}


\maketitle

\section{Introduction} 
  
A classical result of Weyl asserts that the essential spectrum of a bounded linear operator $A$ on some complex separable Hilbert space $\hil$ is invariant under compact perturbations, that is, the essential spectra of $A$ and $B$ coincide if $B-A$ is a compact operator on $\hil$. 
In particular, the discrete eigenvalues of $A$ and $B$ (the isolated eigenvalues of finite algebraic multiplicity), if infinitely many, can accumulate at the joint essential spectrum only. In this paper, we would like to obtain more information on these discrete eigenvalues, and on their rate of accumulation to the essential spectrum, given the stronger assumption that $B-A$ is an element of some von Neumann-Schatten ideal $\mathcal{S}_p$. We recall that a compact operator $K$ on $\hil$ is in $\mathcal{S}_p$, where $p>0$, if 
 $\|K\|_{\mathcal{S}_p}:= \left(\sum_{n} s_n(K)^p\right)^{1/p}$ is finite. Here $s_1(K), s_2(K), \ldots$ are the non-zero  eigenvalues of $|K|:=(K^*K)^{1/2}$ in non-increasing order and counted according to multiplicity. 

To explain what we are aiming for, let us begin by recalling the following result of Kato.
\begin{thm}[Kato, \cite{Kato87}]\label{thm:1}
Let $p\geq 1$ and let $A,B$ be bounded selfadjoint operators on $\hil$ such that $B-A \in \mathcal{S}_p$. Then there exist extended enumerations $\{\alpha_j\}, \{\beta_j\}$ of the discrete eigenvalues of $A,B$, respectively, such that
\begin{equation}\label{eq:1}
  \sum_{j} |\alpha_j-\beta_j|^p \leq \|B-A\|_{\mathcal{S}_p}^p.
\end{equation}  
\end{thm}
Here  an extended enumeration of the discrete eigenvalues is a sequence which contains all the discrete eigenvalues, an eigenvalue of algebraic multiplicity $m$ appearing exactly $m$-times, and which in addition \textit{may} contain boundary points of the essential spectrum. 

It is the aim of this paper to provide a ``weaker'' variant of Kato's theorem, which is valid given the mere assumption that $B-A \in \mathcal{S}_p$ and which, in particular, does not require any additional assumptions (like selfadjointness) on the bounded operators $A$ and $B$. Moreover, we will indicate the usefulness of this variant by applying it to recover and improve some recently established Lieb-Thirring type inequalities for non-selfadjoint Jacobi and Schr\"odinger operators. 

As stated, Kato's theorem need certainly not be true for general $A$ and $B$ (see Example \ref{ex:1} below). However, it is known to remain correct if $A,B$ and $B-A$ are normal, as has been shown by Bhatia and Davis \cite{Bhatia99}, or if $A$ and $B$ are unitary, provided a multiplicative constant $\pi/2$ is added to the right-hand side of (\ref{eq:1}), see Bhatia and Sinha \cite{MR948353}. Some additional known results in the finite-dimensional case can be found in the monographs  \cite{MR1477662} and \cite{MR2325304}.

Since in applications we are usually interested in the discrete eigenvalues of the perturbed operator $B$ only,
as a first candidate for a weaker analog of (\ref{eq:1}) let us consider the estimate 
\begin{equation}
    \label{eq:2}
    \sum_{\lambda \in \sigma_d(B)} \dist(\lambda,\sigma(A))^p \leq \|B-A\|_{\mathcal{S}_p}^p.
\end{equation}
Here $\sigma(A)$ denotes the spectrum of $A$ and  we are summing over all discrete eigenvalues of $B$, where each eigenvalue is counted according to its algebraic multiplicity. This estimate is certainly weaker than estimate (\ref{eq:1}) since for all extended enumerations $\{\alpha_j\}, \{\beta_j\}$ of the discrete eigenvalues of $A$ and $B$ the sum in (\ref{eq:2}) is a lower bound for the sum in (\ref{eq:1}). However, while it has been shown by Bouldin \cite{Bouldin80}  that (\ref{eq:2}) is valid for $p \geq 2$ if $A$ and $B$ are normal operators (without the additional assumption that $B-A$ is normal as well), even in the finite-dimensional case this inequality is far from being true for non-normal operators.
\begin{ex}\label{ex:1}
Let $\hil=\mc^2$ and $p>0$. For $x>0$ define
$$A = \left(
    \begin{array}{cc}
      0 & 1 \\
      0 & 0
    \end{array}\right) \quad \text{and} \quad  B(x) = \left(
    \begin{array}{cc}
      0 & 1 \\
      x & 0
    \end{array}\right).$$ Then $\sigma_d(A)=\{0\}, \sigma_d(B(x))=\{\sqrt{x}, -\sqrt{x}\}$ and 
$$\sum_{\lambda \in \sigma_d(B(x))}\dist(\lambda,\sigma_d(A))^p = 2 x^{p/2}.$$
Moreover, $\|B(x)-A\|_{\mathcal{S}_p}^p=x^p$. But $(2x^{p/2})/x^p=2x^{-p/2}$, which tends to infinity for $x \to 0$.
\end{ex}
 
Hence, in order to obtain a variant of Kato's inequality which remains valid for general $A$ and $B$, inequality (\ref{eq:2}) has to be further weakened. We will see in the next section that the replacement of the spectrum of $A$ by its numerical range is sufficient for this purpose.

\section{An eigenvalue estimate for bounded linear operators}\label{sec:bounded}

In the following let us denote the ideal of bounded linear operators on $\hil$ by $\bdd$ and let us recall that the numerical range of $A \in \bdd$ is defined as $$\num(A)=\{ \langle Af, f \rangle :  f \in \hil, \|f\|=1 \}.$$
It is well known that the spectrum of $A$ is contained in the closure of $\num(A)$ and that $\num(A)$ is always a convex set, see, e.g., \cite{b_Gustafson97}.

The next theorem is our desired variant of Kato's theorem. 
\begin{thm}\label{thm:2}
Let $p\geq1$ and let $A,B \in \bdd$ with $B-A \in \mathcal{S}_p$. Then 
\begin{equation}\label{eq:3}
  \sum_{\lambda \in \sigma_d(B)} \dist(\lambda, \overline{\num}(A))^p \leq \|B-A\|_{\mathcal{S}_p}^p,
\end{equation}  
where each eigenvalue is counted according to its algebraic multiplicity.
\end{thm}
\begin{rem}
  In the following, let us agree that whenever a sum involving eigenvalues is considered, each eigenvalue is counted according to its algebraic multiplicity.
\end{rem}
The short proof of Theorem \ref{thm:2} (and of all other results discussed in this paper) will be presented in Section \ref{sec:proofs}. As we will see, it is a simple adaption of Bouldin's proof of inequality (\ref{eq:2}) for normal operators in \cite{Bouldin80}. Since $\sigma(A) \subset \overline{\num}(A)$ for $A \in \bdd$, we see that estimate (\ref{eq:3}) is indeed weaker than estimate (\ref{eq:2}). In particular, (\ref{eq:3}) only provides information on those eigenvalues of $B$ that are situated outside the numerical range of $A$.

If $A$ is normal, then the closure of its numerical range  coincides with the convex hull of its spectrum (see \cite{b_Gustafson97} Theorem 1.4-4.), so the following corollary is a direct consequence of Theorem \ref{thm:2}. 
\begin{cor}\label{cor:1}
Let $p\geq 1$ and let $A,B \in \bdd$  such that $B-A \in \mathcal{S}_p$. Moreover, let $A$ be normal and let $\sigma(A)$ be convex. Then
\begin{equation}\label{eq:4}
  \sum_{\lambda \in \sigma_d(B)} \dist(\lambda, \sigma(A))^p \leq \|B-A\|_{\mathcal{S}_p}^p.
\end{equation}  
\end{cor}
Estimate (\ref{eq:4}) need not be true if $\sigma(A)$ is a non-convex set, see Remark \ref{rem:new} below. Moreover, we note that on an infinite-dimensional Hilbert space the assumption that $A$ is normal and $\sigma(A)$ is convex already implies that $\sigma(A)=\sigma_{ess}(A)$.
\begin{rem} 
Corollary \ref{cor:1} improves upon a result of Borichev et al., see \cite{Borichev08} Theorem 2.3. In the context of Jacobi operators they showed that  for $A$ selfadjoint with $\sigma(A)=[-2,2]$ the following inequality holds for $p \geq 1$ and every $\eps > 0$ 
 \begin{equation}
   \label{eq:5}
   \sum_{\lambda \in \sigma_d(B)} \frac{\dist(\lambda,[-2,2])^{p+1+\eps}}{|\lambda^2-4|} \leq C(p,\eps,\|B-A\|) \|B-A\|_{\mathcal{S}_p}^p.
 \end{equation}  
The proof of (\ref{eq:5}) uses methods of complex analysis and is thus  completely different from (and more involved than) the method of proof we will use below. In particular, (\ref{eq:4}) is stronger than (\ref{eq:5}) since for $\eps \in (0,1)$ and $\lambda \in \mc \setminus [-2,2]$
$$ \dist(\lambda, [-2,2])^p \geq C(\eps) \frac{\dist(\lambda,[-2,2])^{p+1+\eps}}{|\lambda^2-4|}.$$
  \end{rem}

While we haven't yet touched upon the validity of Kato's theorem and its known generalizations in case that $p \in (0,1)$,  the next example shows that, in the stated generality, (\ref{eq:3}) need not be true in this case.
\begin{ex}
  Let $p \in (0,1)$ and let $\hil=\mc^n$ where $n \geq 2$. Further, let  the $n \times n$-matrices $A$ and $B(x), x > 0,$ be defined by
$$ A = \left(
  \begin{array}{cccc}
    0 & 1 & &   0\\
      & 0 & \ddots &   \\
      &   & \ddots & 1   \\
    0  &   &        & 0  
  \end{array}\right), \quad 
B(x) = A + \left(
  \begin{array}{ccccc}
    0 & 0  & \cdots  & 0 \\
    \vdots  & \vdots    & & \vdots\\
     0 & 0  & \cdots   & 0\\
    x  & 0  & \cdots  & 0 
  \end{array}\right). $$
Then $\num(A)=\{ \lambda : |\lambda| \leq \cos(\pi/(n+1))\}$, see \cite{MR1072339} Proposition 1, and $\sigma_d(B(x))=\{\lambda: \lambda^n=x\}$. Moreover, $\|B(x)-A\|_{\mathcal{S}_p}^p=x^p$ and so for $x \geq [\cos(\pi/(n+1))]^n$  we obtain
$$ \frac{\sum_{\lambda \in \sigma_d(B(x))}\dist(\lambda,\num(A))^p}{\|B(x)-A\|_{\mathcal{S}_p}^p} = \frac{n [ x^{1/n}-\cos(\pi/(n+1))]^p}{x^p}.$$
Regarded as a function of $x$ this quotient takes its maximum value 
$$ \left( \frac{1-1/n}{\cos(\pi/(n+1))} \right)^{p(n-1)}n^{1-p}$$
at $x = \left[ {n(n-1)^{-1} \cos(\pi/(n+1))}\right]^n$. It remains to observe that for $p \in (0,1)$  this maximum value tends to infinity for $n \to \infty$. 
\end{ex}
\begin{rem}\label{rem:new}
Let $A$ and $B(x)$ be defined as in the previous example and set $A'=B(1)$ and $B'=A$. Then $A'$ is normal, $\|{B'}-A'\|_{\mathcal{S}_p}^p=1$ and $\sum_{\lambda \in \sigma_d({B'})} \dist(\lambda,\sigma_d(A'))^p = n$. This shows that Corollary \ref{cor:1} need not be true if the spectrum of the unperturbed operator is non-convex.
\end{rem}
Inequality (\ref{eq:3}) remains true for $p \in (0,1)$ if $A$ and $B$ are selfadjoint operators. This is in contrast to Kato's theorem and inequality (\ref{eq:2}), which, as we will show in the Appendix, need not be true in this case.
\begin{thm}\label{thm:3}
  Let $p \in (0,1)$ and let $A, B \in \bdd$ be selfadjoint operators such that $B-A \in \mathcal{S}_p$. Then
\begin{equation}\label{eq:6}
  \sum_{\lambda \in \sigma_d(B)} \dist(\lambda, \overline{\num}(A))^p \leq \|B-A\|_{\mathcal{S}_p}^p.
\end{equation}  
Here $\overline{\num}(A)= [ \min \sigma(A), \max \sigma(A)]$.   
\end{thm}

\begin{rem}\label{rem:1} 
Let $H_0 \geq 0$  and $H \geq - \omega$ (where $\omega > 0$) be unbounded selfadjoint operators and suppose that $e^{-tH}- e^{-tH_0} \in \mathcal{S}_p$ for some positive $p$ and $t$. Given these assumptions the spectrum of $H$ in $(-\infty,0)$ consists of discrete eigenvalues which can accumulate at $0$ only. As a simple consequence of Theorem \ref{thm:3} and the spectral mapping theorem we obtain that 
\begin{equation}\label{eq:7}
   \sum_{\lambda \in \sigma_d(H)} \dist(e^{-t\lambda}, [0,1])^p \leq \| e^{-tH}-e^{-tH_0} \|_{\mathcal{S}_p}^p.
\end{equation}
In particular, since $e^{x}-1 \geq x$ if $x>0$ we see that 
\begin{equation}\label{eq:8}
   \sum_{\lambda \in \sigma_d(H), \lambda < 0} |\lambda|^p \leq t^{-p} \| e^{-tH}-e^{-tH_0} \|_{\mathcal{S}_p}^p.
\end{equation}
The validity of (\ref{eq:8}) was previously known only with an additional multiplicative constant $C(H,p) \geq 1$ on the right-hand side of the inequality, see Demuth and  Katriel \cite{Demuth08} and Hansmann \cite{Hansmann08}.  We also note that for $p \geq 1$, due to Corollary \ref{cor:1}, inequality (\ref{eq:7}) remains valid when the assumption that $H$ is selfadjoint and semi-bounded is replaced with the assumption that $-H$ is the generator of a strongly continuous semigroup, see \cite{b_Engel00} for the relevant definitions. 
\end{rem}

\section{Application to Schr\"odinger operators}\label{sec:schroedinger}

In this section we apply the results of Section \ref{sec:bounded} to recover and improve some Lieb-Thirring type inequalities for Schr\"odinger operators with complex-valued potentials. For earlier results on this topic we refer to Frank et al. \cite{Frank06}, Bruneau and Ouhabaz \cite{MR2455843}, Laptev and Safronov \cite{MR2540070}, Demuth et al. \cite{DHK08_2} and Safronov \cite{Saf09b}. 

We begin with an application of Theorem \ref{thm:2} to unbounded operators. To this end, let us denote the class of all closed operators in $\hil$ by $\cld$.
 \begin{thm}\label{thm:4}
  Let $H_0 \in \cld$ be selfadjoint with $\sigma(H_0) \subset [0,\infty)$. Moreover, let $H \in \cld$ and assume that $(a+H_0)^{-1}-(a+H)^{-1} \in \mathcal{S}_p$ for some $p \geq 1$ and some $a>0$ with $-a \notin \sigma(H) \cup \sigma(H_0)$. Then 
  \begin{equation}
    \label{eq:9}
\sum_{\lambda \in \sigma_d(H),\Re(\lambda)>0} \frac{|\Im(\lambda)|^p}{|\lambda+a|^{2p}} \leq \|(a+H_0)^{-1}-(a+H)^{-1}\|_{\mathcal{S}_p}^p.
  \end{equation}
\end{thm} 

\begin{rem}
  The assumptions of the previous theorem imply that $\sigma_d(H)$ is contained in $\mc \setminus [0,\infty)$, with possible accumulation points in $[0,\infty)$. We have restricted our consideration to eigenvalues with positive real part in order to avoid some technical difficulties and because this is sufficient for our later purposes.
\end{rem}

In the following, let us set $H=-\Delta + V$ operating in $L^2(\mr^d), d \geq 1$. More precisely, we assume that $V \in L^p(\mr^d)$, where $p > d/2$ if $d \geq 2$ and $p \geq 1$ if $d=1$, and we define $H$ as the unique $m$-sectorial operator associated to the closed, densely defined, sectorial form 
$$\mathcal{E}(f,g)= \langle \nabla f , \nabla g \rangle + \langle Vf,g \rangle, \quad \dom(\mathcal{E})=W^{1,2}(\mr^d).$$
Given these assumptions the resolvent difference $(a+H)^{-1}-(a-\Delta)^{-1}$ is compact for $a$ sufficiently large and so the spectrum of $H$ consists of $[0,\infty)$ and a possible set of discrete eigenvalues, which can accumulate at $[0,\infty)$ only. An application of Theorem \ref{thm:4} will now allow us to prove the following result.

\begin{thm}\label{thm:5}
Let $V \in L^p(\mr^d)$, where $p > d/2$ if $d \geq 2$ and $p \geq 1$ if $d=1$. Moreover, let $\Re(V) \geq 0$. 
Then for every $a > 0$
  \begin{equation}
    \label{eq:10}
 \sum_{\lambda \in \sigma_d(H)} \frac{|\Im(\lambda)|^p}{|\lambda+a|^{2p}} \leq C_0  a^{d/2-2p} \|\Im(V)\|_{L^p}^p.
  \end{equation}
The constant $C_0$ is given explicitly by
\begin{equation}
  \label{eq:11}
 C_0= (2\pi)^{-d} \int_{\mr^d} \frac{dx}{(|x|^2+1)^p}.  
\end{equation}
\end{thm}
\begin{rem}
  The assumption that $\Re(V) \geq 0$ implies that $\sigma(H)$ is contained in the right half-plane.
\end{rem}
Choosing $a=1$ in (\ref{eq:10}) we obtain a slight improvement of a recent result of Laptev and Safronov who showed in \cite{MR2540070}, Theorem $1$, that
\begin{equation}
  \label{eq:12}
\sum_{\lambda \in \sigma_d(H)} \left(\frac{|\Im(\lambda)|}{|\lambda+1|^2+1}\right)^p \leq C_0 \|\Im(V)\|_{L^p}^p.
\end{equation}
More importantly, the free parameter $a$ in (\ref{eq:10}) can be used to derive the following stronger estimate on the eigenvalues of $-\Delta +V$, which takes the behavior of the eigenvalues near $0$ and $\infty$ into better account.
\begin{thm}\label{thm:6}
  Let $V \in L^p(\mr^d)$, where $p > d/2$ if $d \geq 2$ and $p \geq 1$ if $d=1$. Moreover, let $\Re(V) \geq 0$. 
Then for every $\kappa > 0$
  \begin{equation}
    \label{eq:13}
 \sum_{\lambda \in \sigma_d(H), |\lambda|<1} \frac{|\Im(\lambda)|^p}{|\lambda|^{d/2-\kappa}} + \sum_{\lambda \in \sigma_d(H), |\lambda|\geq1} \frac{|\Im(\lambda)|^p}{|\lambda|^{d/2+\kappa}} \leq C_0C_1 \|\Im(V)\|_{L^p}^p.   
  \end{equation}
Here the constant $C_0$ is as defined in (\ref{eq:11}) and $C_1$ is given explicitly by
 \begin{equation*}
 C_1= \left(\int_0^\infty \frac{ds}{s^{1-\kappa}(1+s)^{2\kappa}}\right) \left( \int_0^\infty ds \frac{s^{2p-1-d/2+\kappa}}{(1+s)^{2(p+\kappa)}}\right)^{-1}.   
 \end{equation*}
\end{thm}

\begin{rem}
The choice of the unit disk in (\ref{eq:13}) is somewhat arbitrary and it can be replaced with any other disk centered at zero by a suitable adaption of the constant on the right-hand side. The point is that we obtain a different estimate for eigenvalues accumulating to zero and for eigenvalues accumulating to $(0,\infty)$, respectively.
\end{rem}
If we were allowed to choose $\kappa=0$ in (\ref{eq:13}), then this inequality could be rewritten as 
  \begin{equation}
    \label{eq:14}
     \sum_{\lambda \in \sigma_d(H)} \frac{\dist(\lambda,[0,\infty))^p}{|\lambda|^{d/2}} \leq C(p,d) \|\Im(V)\|_{L^p}^p,   
  \end{equation}
which seems like a natural generalization of the selfadjoint Lieb-Thirring inequality $\sum_{\lambda < 0, \lambda \in \sigma_d(H)} |\lambda|^{p-d/2} \leq C(p,d) \|V_-\|_{L^p}^p$ to the non-selfadjoint setting. Whether (\ref{eq:14}) is indeed valid for Schr\"odinger operators with complex-valued potentials remains an interesting open question.

\section{Application to Jacobi operators}\label{sec:jacobi}

Next, we apply the results of Section \ref{sec:bounded} to recover and improve known Lieb-Thirring type estimates for selfadjoint and non-selfadjoint Jacobi operators. The first results on this topic are due to Hundertmark and Simon \cite{Hundertmark02}, who considered the selfadjoint case. Their results have been extended to  non-selfadjoint operators by Golinskii and Kupin \cite{Golinskii07}, Borichev et al. \cite{Borichev08} and Hansmann and Katriel \cite{HK09}.

Let $\{a_k\}_{k \in \mz}, \{b_k\}_{k \in \mz}$ and $ \{c_k\}_{k \in \mz}$ be bounded complex sequences. The associated Jacobi operator $J=J(a_k,b_k,c_k) : l^2(\mz) \to l^2(\mz)$ is defined as 
$$(J u)(k)=a_{k-1}u(k-1)+b_k u(k)+c_k u(k+1),$$ 
that is, with respect to the standard basis of $l^2(\mz)$ we have
$$ J= \left(
      \begin{array}{ccccccc}
       \ddots & \ddots & \ddots &  &  &  &  \\
        &  a_{-1} & b_0 & c_0 &  &  &   \\
        &  & a_0 & b_1 & c_1 &  &   \\
        &  &  & a_1 & b_2 & c_2 &  \\
         &  &  &  & \ddots & \ddots & \ddots
      \end{array}
    \right).$$
Throughout we assume that $J$ is a compact perturbation of the free Jacobi operator $J_0=J(1,0,1)$, i.e., 
$$(J_0 u)(k)=u(k-1)+ u(k+1).$$
The spectrum  of $J_0$ coincides with the interval
$[-2,2]$ and so the compactness of $J-J_0$ implies   that
$\sigma(J)$ consists of $[-2,2]$ and a possible set of discrete eigenvalues, which can accumulate at $[-2,2]$ only. 
We remark that, defining the sequence $d= \{d_k\}_{k \in \mz} $ as 
\begin{equation}\label{defdk}d_k=\max\Big(|a_{k-1}-1|,|a_{k}-1|,|b_k|,|c_{k-1}-1|,|c_{k}-1|\Big),\end{equation}
the compactness of $J-J_0$ is equivalent to $d_k$ converging to $0$. Moreover, it can be shown that for $p>0$,
\begin{equation}
  \label{eq:15}
\|J-J_0\|_{\mathcal{S}_p} \leq 3 \|d\|_{l^p},  
\end{equation}
see, e.g., \cite{HK09}, Lemma 8. We begin our discussion with the following result for selfadjoint Jacobi operators, which is an immediate consequence of Theorem \ref{thm:3} and estimate (\ref{eq:15}).
\begin{thm}
For $\{a_k\} \subset \mc, \{b_k\} \subset \mr$ and $c_k=\overline{a}_k$  let $J=J(a_k,b_k,c_k)$ be defined as above. Moreover, let $d=\{d_k\}$ be defined by (\ref{defdk}). Then for $p> 0$
  \begin{equation} 
\sum_{\lambda \in \sigma_d(J),\lambda < -2} |\lambda+2|^p + \sum_{\lambda \in \sigma_d(J),\lambda>2} |\lambda-2|^p \leq C(p) \|d\|_{l^p}^p.
  \end{equation}  
\end{thm}
For $p \geq 1$ this inequality is due to Hundertmark and Simon, see \cite{Hundertmark02} Theorem 4. For $p \in (0,1)$ it seems to be new.
\begin{rem}
  Actually, in \cite{Hundertmark02}, Theorem 4, the authors state that their inequality is true for $p\geq 1/2$. However, in their proof they consider the case $p\geq 1$ only.
\end{rem}
Let us consider the  non-selfadjoint case next. As before, the following theorem is a direct consequence of Corollary \ref{cor:1} and estimate (\ref{eq:15}). 
\begin{thm}\label{thm:7}
For $\{a_k\},\{b_k\},\{c_k\} \subset \mc$  let $J=J(a_k,b_k,c_k)$ be defined as above.  Moreover, let $d=\{d_k\}$ be defined by (\ref{defdk}). Then for $p\geq 1$
  \begin{equation}
    \label{eq:16}
\sum_{\lambda \in \sigma_d(J)} \dist(\lambda,[-2,2])^p \leq C(p) \|d\|_{l^p}^p.
  \end{equation}
\end{thm}
Inequality (\ref{eq:16}) improves upon results of Borichev et al., see \cite{Borichev08} Theorem 2.3, and complements a result of Hansmann and Katriel, who showed in \cite{HK09}, Theorem 1, that for $p>1$ and $\eps>0$, 
$$ \sum_{\lambda \in \sigma_d(J)} \frac{\dist(\lambda,[-2,2])^{p+\eps}}{|\lambda^2-4|^{1/2}} \leq C(p,\eps) \|d\|_{l^p}^p,$$
and that a similar estimate, with the exponent $1/2$ replaced by the exponent $1/2+ \eps/4$ in the denominator, is true for $p=1$ as well. Indeed, while the last inequality provides stronger estimates on sequences of eigenvalues converging to $\pm 2$, respectively, inequality (\ref{eq:16}) provides a better estimate on eigenvalues converging to some point in $(-2,2)$. 

\section{Proofs}\label{sec:proofs}

\subsection*{Proof of Theorem \ref{thm:2}}

Let us recall that the $p$th Schatten norm of a linear operator $K$ on $\hil$ can be computed using the following identity, see \cite{b_Simon05} Proposition 2.6.
\begin{lemma}
Let $K \in \mathcal{S}_p$, where $p\geq 1$. Then
\begin{equation}
\label{eq:17}
  \|K\|_{\mathcal{S}_p}^p = \sup_{\{e_n\},\{f_n\}} \left\{ \sum_{n}| \langle K e_n, f_n \rangle|^p \right\},
\end{equation}
where the supremum is taken with respect to arbitrary (finite or infinite) orthonormal sequences $\{e_n\}$ and $\{f_n\}$ in $\hil$.
\end{lemma}
We also need the next result known as Schur's lemma (see, e.g., \cite{b_Gohberg69} Remark 4.1 on page 17).
\begin{lemma}
Let $B \in \bdd$ and let $ \lambda_1, \lambda_2, \ldots $  denote an enumeration of $\sigma_d(B)$, where each eigenvalue is counted according to its algebraic multiplicity and equal eigenvalues are denoted successively. Then there exists an orthonormal sequence $\{e_n\}$ in $\hil$ and a sequence of complex numbers $\{b_{jk}\}$ such that
\begin{equation}
\label{eq:18}
  Be_n=b_{n1}e_1 + b_{n2}e_2 + \ldots + b_{nn}e_n \quad \text{and} \quad   b_{nn}=\lambda_n.
\end{equation}
\end{lemma}
Starting with the proof of Theorem \ref{thm:2}, as in Schur's lemma let $\lambda_1, \lambda_2, \ldots$ denote an enumeration of $\sigma_d(B)$ and  let $\{e_n\}$ be a corresponding orthonormal sequence satisfying (\ref{eq:18}). Then we obtain from (\ref{eq:17}) that for $p \geq 1$
  \begin{equation*}
    \|B-A\|_{S_p}^p \geq \sum_{n} |\langle (B-A) e_n,e_n \rangle|^p = \sum_{n} |\langle Be_n,e_n \rangle - \langle Ae_n,e_n \rangle|^p.
  \end{equation*}
Using (\ref{eq:18})  we have $\langle Be_n, e_n \rangle = \lambda_n$ and so
\begin{eqnarray*}
    \|B-A\|_{S_p}^p \geq \sum_{n} |\lambda_n - \langle Ae_n,e_n \rangle|^p \geq \sum_{n} \dist(\lambda_n,\overline{\num}(A))^p
\end{eqnarray*} 
as desired.

\subsection*{Proof of Theorem \ref{thm:3}} First, let us introduce some notation: If $B \in \bdd$ is selfadjoint then we denote by
$$\lambda_1^- \leq \lambda_2^- \leq \ldots \leq \inf \sigma_{ess}(B)  \quad \text{and} \quad  \lambda_1^+ \geq \lambda_2^+ \geq \ldots \geq \sup \sigma_{ess}(B)$$ the eigenvalues of $B$, counted according to multiplicity, situated below and above the essential spectrum of $B$, respectively. If there exist only $N$ eigenvalues below (above) the essential spectrum, then we set $$\lambda_{N+1}^{-}=\lambda_{N+2}^{-}=\ldots=\inf \sigma_{ess}(B) \quad (\lambda_{N+1}^{+}=\lambda_{N+2}^{+}=\ldots=\sup \sigma_{ess}(B)).$$
\begin{rem}
  Since $\sigma_{ess}(B)$ is empty if $\hil$ is finite-dimensional, we set $\inf \emptyset = + \infty$ and $\sup \emptyset = -\infty$.
\end{rem}

For a proof of the following min-max principle we refer to \cite{b_Reed78}, Theorem XIII.1. 
\begin{prop}\label{prop:1}
  Let $B \in \bdd$ be selfadjoint. Then for every $n \in \mn$ we have 
  \begin{equation*} 
    \lambda_{n}^+= \inf_{\mathcal{W} \subset \hil, \;\dim(\mathcal{W})=n-1} \; \sup_{\psi \in \mathcal{W}^\perp, \|\psi\|=1} \langle B \psi, \psi \rangle
  \end{equation*}
and 
  \begin{equation*}
    \lambda_{n}^-= \sup_{\mathcal{W} \subset \hil, \;\dim(\mathcal{W})=n-1} \; \inf_{\psi \in \mathcal{W}^\perp, \|\psi\|=1} \langle B \psi, \psi \rangle.
  \end{equation*}
\end{prop} 
The next lemma is the main ingredient in the proof of Theorem \ref{thm:3}.
\begin{lemma}\label{lem:45}
Let $A, B \in \bdd$ be selfadjoint with $B-A \in \mathcal{S}_\infty$. Let $a=\inf \num(A)$ and $b=\sup \num(A)$, and let  $\lambda_1^-(B) \leq \lambda_2^-(B) \leq \ldots \leq a$ and $\lambda_1^+(B) \geq \lambda_2^+(B) \geq \ldots \geq b $  denote the eigenvalues of $B$ (counted according to multiplicity) situated  below $a$ and above $b$, respectively. Then 
\begin{equation}
  \label{eq:20}
a-\lambda_n^-(B) \leq s_n((B-A)_-)
\end{equation}
and
\begin{equation}
  \label{eq:21}
  \lambda_n^+(B)-b \leq s_n((B-A)_+). 
\end{equation}
\end{lemma} 
\begin{rem}
We note that $\sigma_{ess}(A)=\sigma_{ess}(B)$ since $B-A$ is compact. Hence, we have
$\inf \num(A) \leq \inf \sigma_{ess}(B)$ {and} $\sup \num(A) \geq \sup \sigma_{ess}(B)$. 
Moreover, we recall that $$(B-A)_\pm = 1 / 2 \left( |B-A| \pm (B-A) \right)$$ denote the positive and negative parts of $B-A$, respectively. 
\end{rem}
\begin{proof}[Proof of Lemma \ref{lem:45}] 
Since ${A} \geq a$, and so $(B-A)_- \geq A-B \geq a-B$, we obtain from Proposition \ref{prop:1} that 
\begin{eqnarray*}
      a-\lambda_{n}^-(B) &=& a- \sup_{\mathcal{W} \subset \hil, \;\dim(\mathcal{W})=n-1} \; \inf_{\psi \in \mathcal{W}^\perp, \|\psi\|=1} \langle B \psi, \psi \rangle \\
&=& \inf_{\mathcal{W} \subset \hil, \;\dim(\mathcal{W})=n-1} \; \sup_{\psi \in \mathcal{W}^\perp, \|\psi\|=1} \langle (a-B) \psi, \psi \rangle \\
&\leq& \inf_{\mathcal{W} \subset \hil, \;\dim(\mathcal{W})=n-1} \; \sup_{\psi \in \mathcal{W}^\perp, \|\psi\|=1} \langle (B-A)_- \psi, \psi \rangle \\
&=& \lambda_n^+((B-A)_-) = s_n((B-A)_-).
\end{eqnarray*}
The proof of (\ref{eq:21}) is completely analogous and is therefore omitted.
\end{proof}
We are now prepared for the proof of Theorem \ref{thm:3}: Let  $a=\inf \num(A)$ and $b= \sup \num (A)$. Since for $p > 0$ we can write
$$ \sum_{\lambda \in \sigma_d(B)} \dist(\lambda, \overline{\num}(A))^p = \sum_{\lambda \in \sigma_d(B), \lambda < a} (a-\lambda)^p + \sum_{\lambda \in \sigma_d(B), \lambda > b} (\lambda-b)^p,$$
the validity of (\ref{eq:6}) is a consequence of Lemma \ref{lem:45} and the fact that
$$ \| (B-A)_- \|_{\mathcal{S}_p}^p + \| (B-A)_+\|_{\mathcal{S}_p}^p= \|B-A\|_{\mathcal{S}_p}^p .$$

\subsection*{Proof of Theorem \ref{thm:4}}
Applying  Theorem \ref{thm:2} with $A=(a+H_0)^{-1}$ and $B=(a+H)^{-1}$ we obtain
\begin{equation*}
  \sum_{\mu \in \sigma_d((a+H)^{-1})} \dist(\mu, \overline{\num}((a+H_0)^{-1}))^p \leq \|(a+H)^{-1}-(a+H_0)^{-1}\|_{\mathcal{S}_p}^p.
\end{equation*}
Since $\sigma(H_0) \subset [0,\infty)$ by assumption, the spectral mapping theorem implies that $\sigma((a+H_0)^{-1}) \subset [0,a^{-1}]$. But then $\overline{\num}((a+H_0)^{-1}) \subset [0,a^{-1}]$ as well, since the numerical range of a bounded selfadjoint operator coincides with the convex hull of its spectrum. Using the spectral mapping theorem once again we thus see that the sum in the previous estimate is bounded from below by
\begin{equation*}
    \sum_{\lambda \in \sigma_d(H)} \dist((a+\lambda)^{-1},[0,a^{-1}])^p.
\end{equation*}
The proof is completed by noting that for $\Re(\lambda)>0$
\begin{equation}
  \label{eq:22}
 \dist((a+\lambda)^{-1},[0,a^{-1}]) = \frac{|\Im(\lambda)|}{|\lambda+a|^2},  
\end{equation}
as a short computation shows.

\subsection*{Proof of Theorem \ref{thm:5}}
We start with a lemma.
\begin{lemma}
  Let $V \in L^p(\mr^d)$, where $p > d/2$ if $d \geq 2$ and $p \geq 1$ if $d=1$. Moreover, let $V_n \in L^p(\mr^d)$  be chosen such that $\|V_n-V\|_{L^p} \to 0$ for $n \to \infty$. Let $K \subset \mc \setminus [0,\infty)$ be compact and let $f : \mc \setminus [0,\infty) \to \mr_+$ be continuous. Then
  \begin{equation}
    \label{eq:23}
    \sum_{\lambda \in \sigma_d(-\Delta+V) \cap K} f(\lambda) = \lim_{n \to \infty} \sum_{\lambda \in \sigma_d(-\Delta+V_n) \cap K} f(\lambda), 
  \end{equation}
where in each sum every eigenvalue is counted according to its algebraic multiplicity.
\end{lemma}
\begin{proof}
We only provide a sketch. More details can be found in \cite{MR2540070}, Section 2. Since $\|V_n - V\|_{L^p} \rightarrow 0$ for $n \to \infty$, we obtain that $$\|(a+(-\Delta+V_n))^{-1} - (a+(-\Delta+V))^{-1}\| \rightarrow 0$$ for $n \to \infty$ and $a>0$ sufficiently large, as, for instance, is a consequence of \cite{MR1335452}, Theorem VI.3.5. In particular, this implies that each eigenvalue  of $-\Delta + V$ is the limit of a sequence of eigenvalues  of $-\Delta +V_n$, with the multiplicities being preserved in the obvious manner, see the discussion on page 213 in \cite{MR1335452}. Since every compact set $K \subset \mc \setminus [0,\infty)$ can contain only finitely many eigenvalues of $-\Delta+V$, the continuity of $f$ implies the validity of (\ref{eq:23}).
\end{proof} 
Due to the previous lemma it is sufficient to prove Theorem \ref{thm:5} for $V \in C_0^\infty(\mr^d)$. In the following, let $H=-\Delta+V$ and set $H_0=-\Delta+\Re(V)$. Both $H$ and $H_0$ are defined on $\dom(-\Delta)=W^{2,2}(\mr^d)$. Moreover, $\sigma(H_0) \subset [0,\infty)$ and $\sigma(H) \subset \{ \lambda : \Re(\lambda) \geq 0 \}$ since $\Re(V) \geq 0$ by assumption. Hence, Theorem \ref{thm:4} implies that
\begin{equation}
  \label{eq:24}
\sum_{\lambda \in \sigma_d(H)} \frac{|\Im(\lambda)|^p}{|\lambda+a|^{2p}} \leq  \|(a+H_0)^{-1}-(a+H)^{-1}\|_{\mathcal{S}_p}^p,
\end{equation}
provided the right-hand side of (\ref{eq:24}) is finite. To obtain an estimate on this Schatten norm we adapt an idea used in \cite{MR2540070}. First, we rewrite the resolvent difference as 
\begin{eqnarray*}
 && (a+H_0)^{-1}-(a+H)^{-1} \nonumber \\
&=& i (a+H)^{-1}(a-\Delta)^{1/2}(a-\Delta)^{-1/2} |W|^{1/2} \operatorname{sign}(W) \nonumber \\
&& |W|^{1/2}(a-\Delta)^{-1/2}(a-\Delta)^{1/2}(a+H_0)^{-1}, 
\end{eqnarray*}
where $W=\Im(V)$ and $\operatorname{sign}(W)=W/|W|$. We will show below that $(a-\Delta)^{1/2}(a+H_0)^{-1}$ is bounded on $\hil$ with
\begin{equation}
  \label{eq:25}
\|(a-\Delta)^{1/2}(a+H_0)^{-1}\| \leq a^{-1/2}  
\end{equation}
and that the same is true for the closure of $(a+H)^{-1}(a-\Delta)^{1/2}$, initially defined on $\operatorname{Dom}((-\Delta)^{1/2})=W^{1,2}(\mr^d)$, i.e.,
\begin{equation}
  \label{eq:26}
  \|\overline{(a+H)^{-1}(a-\Delta)^{1/2}}\| \leq a^{-1/2}.
\end{equation}
Hence,
\begin{eqnarray*}
&& \|(a+H_0)^{-1}-(a+H)^{-1}\|_{\mathcal{S}_p}^p \\
&\leq&  a^{-p} \|(a-\Delta)^{-1/2} |W|^{1/2} \operatorname{sign}(W)|W|^{1/2}(a-\Delta)^{-1/2} \|_{\mathcal{S}_p}^p \\
&\leq& a^{-p} \|(a-\Delta)^{-1/2} |W|^{1/2}\|_{\mathcal{S}_{2p}}^{2p},
\end{eqnarray*}
where for the last estimate we used H\"older's inequality for Schatten norms, see \cite{b_Simon05} Theorem 2.8. Since $p\geq1$ and $p>d/2$, the proof is concluded by an application of Theorem 4.1 from \cite{b_Simon05}, i.e.,
$$ \|(a-\Delta)^{-1/2} |W|^{1/2}\|_{\mathcal{S}_{2p}}^{2p} \leq (2\pi)^{-d} \|(a+|.|^2)^{-1}\|_{L^p}^p \|W\|_{L^p}^p.$$
It remains to prove (\ref{eq:25}) and (\ref{eq:26}). To this end, let $f \in L^2(\mr^d)$. Then
  \begin{eqnarray*}
  &&  \|(a-\Delta)^{1/2}(a+H^*)^{-1}f\|^2 \\
&=& \langle f, (a+H^*)^{-1}f \rangle - \langle \overline{V} (a+H^*)^{-1}f, (a+H^*)^{-1}f \rangle.
  \end{eqnarray*}
Since $\Re(V) \geq 0$ we obtain
  \begin{eqnarray}
  &&  \|(a-\Delta)^{1/2}(a+H^*)^{-1}f\|^2 \nonumber \\
&=& \Re(\langle f, (a+H^*)^{-1}f \rangle) - \Re(\langle \overline{V} (a+H^*)^{-1}f, (a+H^*)^{-1}f \rangle) \nonumber \\
&\leq& \Re(\langle f, (a+H^*)^{-1}f \rangle) \leq |\langle f, (a+H^*)^{-1}f \rangle| \nonumber \\
&\leq& \|f\|^2\|(a+H)^{-1}\| \leq \frac{ \|f\|^2}{\dist(-a,\overline{\num}(H))}\leq a^{-1} \|f\|^2, \label{eq:27}
  \end{eqnarray}
where for the last two estimates we used that $H$ is $m$-sectorial with $\overline{\num}(H) \subset \{ \lambda : \Re(\lambda) \geq 0\}$. But (\ref{eq:27}) implies (\ref{eq:26}) since $$\overline{(a+H)^{-1}(a-\Delta)^{1/2}}=[(a-\Delta)^{1/2}(a+H^*)^{-1}]^*.$$ The proof of (\ref{eq:25}) is similar (and even simpler) and is therefore omitted.

\subsection*{Proof of Theorem \ref{thm:6}}
Using that $|\lambda+a| \leq |\lambda|+a$ we obtain from (\ref{eq:10}) that
\begin{equation*}
  \sum_{\lambda \in \sigma_d(H)} \frac{|\Im(\lambda)|^p a^{2p-d/2-1+\kappa}}{(|\lambda|+a)^{2p}(1+a)^{2\kappa}} \leq C_0 \frac{\|\Im(V)\|_{L^p}^p}{a^{1-\kappa}(1+a)^{2\kappa} }.
\end{equation*}
Integrating both sides of this inequality with respect to $a \in (0,\infty)$ leads to
\begin{eqnarray}
&&  \sum_{\lambda \in \sigma_d(H)} |\Im(\lambda)|^p \int_0^\infty da \: \frac{ a^{2p-d/2-1+\kappa}}{(|\lambda|+a)^{2p}(1+a)^{2\kappa}} \nonumber \\
&\leq& C_0 \|\Im(V)\|_{L^p}^p \int_0^\infty \frac{da}{a^{1-\kappa}(1+a)^{2\kappa} }.  \label{eq:28}
\end{eqnarray}
Here the right-hand side is finite since $\kappa > 0$. Substituting $|\lambda| s =a$, the integral on the left-hand side of (\ref{eq:28}) can be rewritten as
\begin{equation}
\label{eq:29}
    \frac{1}{|\lambda|^{d/2-\kappa}} \int_0^\infty ds\: \frac{s^{2p-d/2-1+\kappa}}{(1+s)^{2p}(1+|\lambda|s)^{2\kappa}}
\end{equation}
and this integral is bounded from below by
\begin{equation}
  \label{eq:30}
      \frac{1}{|\lambda|^{d/2-\kappa}\operatorname{max}(|\lambda|,1)^{2\kappa}} \int_0^\infty ds \: \frac{s^{2p-d/2-1+\kappa}}{(1+s)^{2p+2\kappa}}.
\end{equation}
Now (\ref{eq:28})-(\ref{eq:30}) imply the validity of (\ref{eq:13}).

\section{Appendix}

The following proposition and its corollary imply that inequality (\ref{eq:2}), and hence Kato's theorem, need not be true for selfadjoint operators $A,B$ in case that $p \in (0,1)$.

\begin{prop}\label{prop:2}
 Let $A,B$ be hermitian matrices on $\mc^n$. Then for $p \in (0,1)$
\begin{equation}
  \label{eq:31}
 \sum_{\lambda \in \sigma_d(B)} \dist(\lambda, \sigma_d(A))^p  \leq n^{1-p} \|B-A\|_{\mathcal{S}_p}^p.  
\end{equation} 
Moreover, the constant $n^{1-p}$ cannot be replaced with any smaller constant.
\end{prop}

\begin{cor}
Let $p \in (0,1)$. Then there does not exist a constant $C(p)>0$ such that for every complex separable Hilbert space $\hil$ and all selfadjoint bounded operators $A,B$ on $\hil$ with $B-A \in \mathcal{S}_p$ we have
$$ \sum_{\lambda \in \sigma_d(B)} \dist(\lambda,\sigma(A))^p \leq C(p) \|B-A\|_{\mathcal{S}_p}^p.$$
\end{cor}

\begin{proof}[Proof of Proposition \ref{prop:2}]
Let $p \in (0,1)$. First, we note that $ \|B-A\|_{\mathcal{S}_1} \leq \|B-A\|_{\mathcal{S}_p}$. Moreover,  the $l^1$-norm of a vector $x \in \mr^n$ is bounded from below by $n^{1-1/p}$ times the $l^p$-(quasi-)norm of $x$, as is a consequence of the inverse H\"older inequality. Put together the last two remarks and the $p=1$ case of inequality (\ref{eq:2}) imply the validity of (\ref{eq:31}).

Now let $\{\alpha_k\}_{k=1}^n$ be any sequence of pairwise distinct real numbers. Let $e_1,\ldots,e_n$ denote the standard basis of $\mc^n$ and define the hermitian diagonal matrix $A$ by setting $Ae_k=\alpha_ke_k$ for $k=1,\ldots,n$. Further, define the hermitian diagonal matrix $D$ by setting $De_1=e_1$ and $De_k=0$ for $k=2,\ldots,n$. Finally, let $U=(u_1,\ldots,u_n)$ be any unitary matrix on $\mc^n$, the $u_k$'s denoting its orthonormal column vectors, and for $x > 0$ put $B(x)=A+ x U^*DU$, so that $B(x)$ is hermitian as well.  Since the $e_k$'s are precisely the eigenvectors of $A$ corresponding to the eigenvalues $\alpha_k$, standard perturbation theory (see e.g. \cite{b_Reed78}, Section XII.1) implies that for $x$ small enough there exist $n$ pairwise distinct eigenvalues $\lambda_k(x)$ of $B(x)$ obeying
  \begin{eqnarray*} 
 \lambda_k(x)&=&\alpha_k+ x \langle U^*DUe_k,e_k \rangle + O(x^2) \\
             &=& \alpha_k + x |u_{k,1}|^2 +O(x^2),
  \end{eqnarray*}
where $u_k=(u_{k,1},\ldots, u_{k,n})^T$ and $k=1,\ldots,n.$
Hence, for $x$ small enough
$$ \sum_{\lambda \in \sigma_d(B(x))} \dist(\lambda,\sigma_d(A))^p = x^p \sum_{k=1}^n  |u_{k,1}|^{2p} + O(x^{2p}).$$
Regarding the $\mathcal{S}_p$-norm of $B(x)-A$, we have 
$$ \|B(x)-A\|_{\mathcal{S}_p}^p =x^p \|U^*DU\|_{\mathcal{S}_p}^p=x^p \|D\|_{\mathcal{S}_p}^p=x^p,$$
so for $x \to 0$ 
$$ \frac{ \sum_{\lambda \in \sigma_d(B(x))} \dist(\lambda,\sigma_d(A))^p }{\|B(x)-A\|_{\mathcal{S}_p}^p} \quad {\rightarrow} \quad \sum_{k=1}^n  |u_{k,1}|^{2p}.$$
It remains to note that we can always find orthonormal vectors $u_k$ such that $u_{k,1}=1/\sqrt{n}$ for every $k = 1,\ldots, n$. For instance, we can choose $  u_1 = (1/\sqrt{n}, \sqrt{1-1/n}, 0, \ldots, 0 )^T$ and define recursively
\begin{eqnarray*}
   u_k &=& \left(
   \begin{array}{c}
     u_{k-1,1} \\
     \vdots \\
     u_{k-1,k-1} \\
     - \left( \sum_{j=1}^{k-1} u_{k-1,j}^2 \right)/ u_{k-1,k} \\
     \sqrt{1-\sum_{j=1}^k u_{k,j}^2} \\
     0 \\
     \vdots \\
     0 
   \end{array} \right), \quad k=2,\ldots,n-1
\end{eqnarray*}
and $u_n = \left( u_{n-1,1}, \ldots, u_{n-1,n-1}, - u_{n-1,n} \right)^T$.

\end{proof}

\section*{Acknowledgments}

The author would like to thank M.~Demuth, M.~J.~Gruber and G.~Katriel for their valuable comments.
 
\bibliography{hansmann_bibliography} 
\bibliographystyle{plain} 

\end{document}